\documentclass[12pt]{amsart}

\newcommand{\bw}{\beta(w)}

\newenvironment{smat}{\left(\begin{smallmatrix}}
{\end{smallmatrix}\right)}

\renewcommand{\b}{\mathbf{b}}

\newcommand{\dsum}{\di\sum}

\newcommand{\dep}[1]{{\vrule width 0pt height 0pt depth #1}}

\newcommand{\down}{\downarrow}

\newcommand{\el}{\ensuremath{\ell}}

\DeclareFontFamily{U}{mathx}{\hyphenchar\font45}
\DeclareFontShape{U}{mathx}{m}{n}{
      <5> <6> <7> <8> <9> <10>
      <10.95> <12> <14.4> <17.28> <20.74> <24.88>
      mathx10
      }{}
\DeclareSymbolFont{mathx}{U}{mathx}{m}{n}
\DeclareFontSubstitution{U}{mathx}{m}{n}
\DeclareMathAccent{\widecheck}{0}{mathx}{"71}

\newcommand{\tw}{\textwidth}

\renewcommand{\kill}[1]{}
\newcommand{\dummy}[1]{\mbox{}}

\makeatletter
\newcommand{\xequal}[2][]{\ext@arrow 0055{\equalfill@}{#1}{#2}}
\def\equalfill@{\arrowfill@\Relbar\Relbar\Relbar}
\makeatother

\newcommand{\mto}{\mapsto}

\newcommand{\1}{\ensuremath{\ol{\mathrm{P}}}}

\renewcommand{\k}{\ensuremath{\ol{\mathrm{P}}}}

\newcommand{\hl}{\hline}

\newcommand{\h}{\hline}

\renewcommand{\k}[1]{\ensuremath{\left({#1}\right)}}

\newcommand{\ds}{\dots}

\newcommand{\bca}{\begin{cases}}
\newcommand{\eca}{\end{cases}}

\newcommand{\A}{\mathcal{A}}

\newcommand{\ff}[2]{\ensuremath{\di\fr{#1}{#2}}}

\newcommand{\s}[1]{\ensuremath{\di\int{#1}\,dx}}

\newcommand{\bpic}{\begin{picture}}\newcommand{\epic}{\end{picture}}

\newcommand{\beda}{\begin{edaenumerate}}
\newcommand{\eeda}{\end{edaenumerate}}

\newcommand{\cd}{\cdots}

\newcommand{\st}{\strut}
\newcommand{\mst}{\mathstrut}

\newcommand{\q}{\quad}

\newcommand{\bq}{\begin{quote}}\newcommand{\eq}{\end{quote}}

\newcommand{\vi}{\\[.1in]}

\newcommand{\be}{\begin{enumerate}}\newcommand{\ee}{\end{enumerate}}
\newcommand{\bce}{\begin{center}}\newcommand{\ece}{\end{center}}
\newcommand{\bde}{\begin{description}}\newcommand{\ede}{\end{description}}
\newcommand{\bri}{\begin{flushright}}\newcommand{\eri}{\end{flushright}}
\newcommand{\bb}{\begin{block}}\newcommand{\eb}{\end{block}}
\newcommand{\bt}{\begin{thm}}\newcommand{\et}{\end{thm}}
\newcommand{\bpf}{\begin{proof}}\newcommand{\epf}{\end{proof}}
\newcommand{\bex}{\begin{ex}}\newcommand{\eex}{\end{ex}}
\newcommand{\bexr}{\begin{exr}}\newcommand{\eexr}{\end{exr}}
\newcommand{\bft}{\begin{fact}}\newcommand{\eft}{\end{fact}}
\newcommand{\brk}{\begin{rmk}}\newcommand{\erk}{\end{rmk}}
\newcommand{\ba}{\begin{align*}}\newcommand{\ea}{\end{align*}}
\newcommand{\bexe}{\begin{exe}}\newcommand{\eexe}{\end{exe}}
\newcommand{\tn}{\textnormal}

\newcommand{\bit}{\begin{itemize}}\newcommand{\eit}{\end{itemize}}

\newcommand{\bcm}{}
\newcommand{\ol}{\overline}\newcommand{\ul}{\underline}
\newcommand{\hf}{\hfill}
\newcommand{\fr}{\frac}

\newcommand{\bd}{\begin{defn}}\newcommand{\ed}{\end{defn}}
\newcommand{\bp}{\begin{prop}}\newcommand{\ep}{\end{prop}}

\newcommand{\eh}{\emph}
\newcommand{\sub}{\subseteq}
\newcommand{\lam}{\lambda}
\newcommand{\fb}{\fbox}
\newcommand{\mb}{\mbox}
\newcommand{\te}{\text}\newcommand{\ph}{\phantom}
\newcommand{\wt}{\widetilde}

\renewcommand{\l}{\left}

\newcommand{\leftexp}[2]{{\vphantom{#2}}^{#1}{#2}}
\newcommand{\di}{\displaystyle}\renewcommand{\a}{\ensuremath{\bm{a}}}
\renewcommand{\b}{\ensuremath{\bm{b}}}

\renewcommand{\b}{\beta}

\renewcommand{\a}{\alpha}

\renewcommand{\int}{\in T}

\renewcommand{\s}{\sigma}

\usepackage[dvipdfmx]{graphicx}
\usepackage[dvipsnames]{xcolor}
\usepackage{float,afterpage}
\usepackage{asymptote,layout}
\usepackage{wrapfig,epic}

\usepackage{geometry}
\usepackage{exscale,latexsym,bm}
\usepackage{amssymb,enumerate,amsmath,amsthm,amsfonts}
\usepackage{verbatim,fancybox,wasysym,fancyhdr,type1cm}
\usepackage[frame,all,poly,curve,knot,arrow]{xy}
\usepackage{colortbl}
\usepackage{boxedminipage}
\usepackage{multirow}

\graphicspath{{./figs/}}
\theoremstyle{definition}
\newtheorem{thm}{Theorem}[section]
\newtheorem{lem}[thm]{Lemma}
\newtheorem{prop}[thm]{Proposition}\newtheorem{cor}[thm]{Corollary}

\newtheorem{exr}[thm]{Exercise}

\newtheorem{open}[thm]{Open Problem}

\newtheorem{ex}[thm]{Example}

\newtheorem{defn}[thm]{Definition}\newtheorem{rmk}[thm]{Remark}
\newtheorem*{notn}{Notation}\newtheorem{fact}[thm]{Fact}
\newtheorem{block}[thm]{}
\newtheorem*{exe}{Exercise}

\newtheorem*{thm1}{Theorem \ref{mth0}}

\title[weighted counting of inversions on ASMs]{weighted counting of inversions on alternating sign matrices}
\author[Masato Kobayashi]{Masato Kobayashi$^{*}$}
\date{\today}
\address{Department of Engineering\\
Kanagawa University, 3-27-1 Rokkaku-bashi, Yokohama 221-8686, Japan.}
\thanks{*Department of Engineering, Kanagawa University, Japan}
\thanks{This article is to appear in Order.}
\email{masato210@gmail.com}
\begin{document}
\newcommand{\self}{\tn{side}_{w}}
\newcommand{\md}{\tn{mid}}
\newcommand{\ilj}{\ensuremath{\leftexp{I}{\el}^{J}}}
\newcommand{\ipj}{\ensuremath{\leftexp{I}{\mathstrut P}^{J}}}
\newcommand{\pd}{P_{\down}}
\begin{abstract} 
We generalize the author's formula (2011) on weighted counting of inversions on permutations to one on alternating sign matrices. The proof is based on the sequential construction of alternating sign matrices from the unit matrix which essentially follows from 
the earlier work of Lascoux-Sch\"{u}tzenberger (1996).
\end{abstract}
\maketitle
\tableofcontents

\section{Introduction}
\subsection{Alternating sign matrices and others}

\emph{Alternating sign matrices} (ASMs) are one of the most important  topics in combinatorics. There is the long story \cite{bressoud} to the proofs by Zeilberger and Kuperberg (both 1996) on the total number of ASMs of size $n$ being 
\[
\prod_{i=0}^{n-1}\ff{(3i+1)!}{(n+i)!}.
\]
As Propp mentioned \cite{propp}, there are actually many combinatorial objects which are equinumerous with ASMs (Figure \ref{f1}):
\begin{itemize}
	\item Monotone triangle
	\item DPP (Descending plane partition)
	\item Square ice
	\item FPL (Full packing of loops)
\end{itemize}
These are numbers in an array or graphs on a grid satisfying certain conditions (we omit details). Among these, we show that entries $\{-1, 0, 1\}$ of ASMs
are numerically meaningful with the connection to its poset structures 
and \eh{inversions} often discussed on permutations.

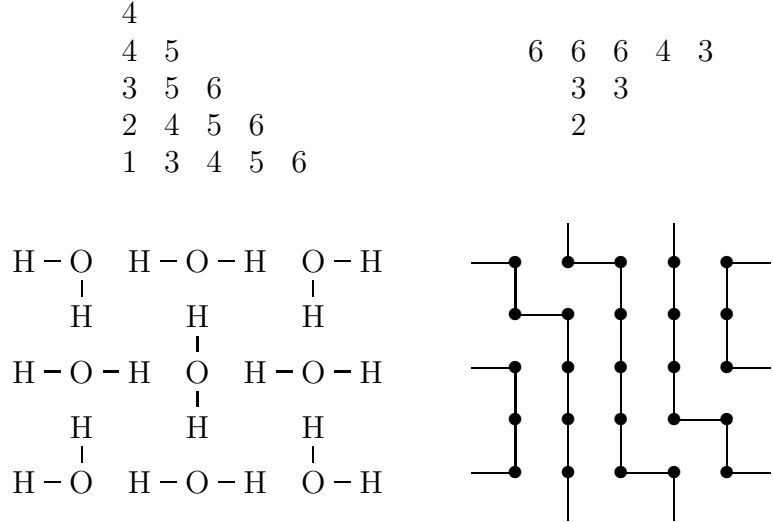
\begin{figure}
\caption{Monotone triangle, DPP, square ice, FPL}
\label{f1}
\begin{center}
\begin{tabular}{cccccccc}
&&\\
\begin{minipage}[c]{.3\tw}
\mb{}\hf$\begin{array}{cccccccc}
	4&\\
	4&5    \\
	3& 5  &6\\
	2& 4  &5&6\\
	1& 3  &4&5&6\\
\end{array}$\hf\mb{}
\end{minipage}
	&
\begin{minipage}[c]{.3\tw}
\mb{}\hf$\begin{array}{cccccccc}
	6&6   &6 &4&3  \\
	&3   &3   \\
	& 2  &   
\end{array}$\hf\mb{}
\end{minipage}
	\\[.5in]
\begin{minipage}[c]{.3\tw}
\xymatrix@=2mm{
*+{\te{H}}\ar@{-}[r]&*+{\te{O}}\ar@{-}[d]&
*+{\te{H}}\ar@{-}[r]&*+{\te{O}}\ar@{-}[r]&
*+{\te{H}}&*+{\te{O}}\ar@{-}[r]\ar@{-}[d]&*+{\te{H}}&\\
&*+{\te{H}}&&*+{\te{H}}\ar@{-}[d]&&
*+{\te{H}}&\\
*+{\te{H}}\ar@{-}[r]&*+{\te{O}}\ar@{-}[r]&
*+{\te{H}}&*+{\te{O}}\ar@{-}[d]&*+{\te{H}}\ar@{-}[r]
&*+{\te{O}}\ar@{-}[r]&*+{\te{H}}\\
&*+{\te{H}}\ar@{-}[d]&&*+{\te{H}}&&*+{\te{H}}&\\
*+{\te{H}}\ar@{-}[r]&*+{\te{O}}&
*+{\te{H}}\ar@{-}[r]&
*+{\te{O}}\ar@{-}[r]&
*+{\te{H}}&
*+{\te{O}}\ar@{-}[r]\ar@{-}[u]&
*+{\te{H}}
}	
\end{minipage}
	&
\begin{minipage}[c]{.3\tw}
	\xymatrix@=4.5mm{
&&&&&&\\
&*-{\bullet}\ar@{-}[l]\ar@{-}[d]&*-{\bullet}\ar@{-}[u]\ar@{-}[r]&*-{\bullet}\ar@{-}[d]
&*-{\bullet}\ar@{-}[u]\ar@{-}[d]&*-{\bullet}\ar@{-}[r]\ar@{-}[d]&\\
&*-{\bullet}&*-{\bullet}\ar@{-}[l]\ar@{-}[d]&*-{\bullet}&*-{\bullet}&*-{\bullet}&\\
&*-{\bullet}\ar@{-}[l]\ar@{-}[d]&*-{\bullet}&*-{\bullet}\ar@{-}[u]\ar@{-}[d]&*-{\bullet}\ar@{-}[u]\ar@{-}[d]
&*-{\bullet}\ar@{-}[u]\ar@{-}[r]&\\
&*-{\bullet}&*-{\bullet}\ar@{-}[u]\ar@{-}[d]&*-{\bullet}&*-{\bullet}&*-{\bullet}\ar@{-}[l]\ar@{-}[d]&\\
&*-{\bullet}\ar@{-}[l]\ar@{-}[u]&*-{\bullet}\ar@{-}[d]&*-{\bullet}\ar@{-}[r]\ar@{-}[u]&
*-{\bullet}\ar@{-}[d]&*-{\bullet}\ar@{-}[r]&\\
&&&&&&\\
}
\end{minipage}
\\		
\end{tabular}
\end{center}
\end{figure}



{\renewcommand{\arraystretch}{2}
\begin{table}[h!]
\caption{Inversion and $\beta$}
\label{t1}
\begin{center}
	\begin{tabular}{|c|c|c|ccccc}\h
	&permutation $w\in S_{n}$	&ASM $A\in \A_{n}$	\\\h
inversion	&
\rule{0pt}{25pt}\dep{25pt}
$I(w)=\dsum_{i<j, w(i)>w(j)}1$
	&$I(A)=\dsum_{i<j, k<l}a_{jk}a_{il}$
		\\\h
$\beta$	&\rule{0pt}{25pt}\dep{25pt}
	$\beta(w)=\di\sum_{i<j, w(i)>w(j)}(w(i)-w(j))$ &	
	See Theorem \ref{mth0} 
	and Remark \ref{this}
	\\\h
\end{tabular}
\end{center}
\end{table}}

\subsection{Main result}

We say $(i, j)$ is an \emph{inversion} of a permutation $w$ of $\{1, 2, \ds, n\}$ if $i<j$ and $w(i)>w(j)$. 
The \emph{inversion number} of $w$ is the number of such pairs, i.e., 
\[
I(w)=|\{(i, j)\mid i<j \te{ and } w(i)>w(j)\}|.
\]
This idea plays a key role on \eh{Bruhat order} in the theory of Coxeter groups (which has also a wide variety of applications in combinatorics and other areas). 
\begin{defn}
Define \eh{Bruhat order} $v\le w$ in $S_{n}$ (the symmetric group) if 
there exists a sequence of permutations 
$v=v_{0}, v_{1}, \ds, v_{m}=w$ such that 
$v_{k+1}=v_{k}t_{i_{k}j_{k}}$ ($t_{i_{k}j_{k}}$ a  transposition) 
and $I(v_{k})<I(v_{k+1})$ for each $k$. Say a permutation $v\in S_{n}$ is \emph{bigrassmannian} if there exists a unique pair $(i, j)$ such that $v^{-1}(i)>v^{-1}(i+1)$ and $v(j)>v(j+1)$. 
Define $\beta:S_{n}\to \{0, 1, 2, \ds\}$ by 
\[
\beta(w)=|\{v\in S_{n} \mid v\le w \te{ and } v \te{ is bigrassmannian}\}|.
\]
\end{defn}
The author \cite[Theorem]{kob} showed that 
\[
\beta(w)=\di\sum_{i<j, w(i)>w(j)}(w(i)-w(j))
\]
for all $w\in S_{n}$. 
We can interpret this formula as weighted counting of inversions on permutations. 
We can extend this statistic $\beta$ for alternating sign matrices as follows: 
\bd{Let $A=(a_{ij})$ be a square matrix of size $n$. We say that $A$ is an \emph{alternating sign matrix} (ASM) if
for all $i, j$, 
we have 
\[\begin{array}{lll}\di a_{ij}\in \{-1, 0, 1\},  &   &\di\sum_{k=1}^j a_{ik}\in \{0, 1\},   \vi \di\sum_{k=1}^i a_{kj}\in \{0, 1\} &\te{ and }   &  \di\sum_{k=1}^n a_{ik}= \sum_{k=1}^n a_{kj}=1.\end{array}\]

Denote by $\A_n$ the set of all alternating sign matrices of size $n$.
}\ed
There is a well-known identification of permutations and permutation matrices:
 For $w\in S_{n}$, let 
\[
w_{ij}=\delta_{jw(i)}=
\begin{cases}
	1&j=w(i),\\
	0& \tn{ otherwise.}\\
\end{cases}
\]
Then, $w=(w_{ij})_{i, j=1}^{n}$ is an ASM with entries only $\{0, 1\}$. Under this identification, we can naturally think $S_{n}\sub \A_{n}$.
\bd{The \emph{corner sum matrix} of $A\in \A_n$ is the $n$ by $n$ matrix $\wt{A}$ defined by 
\[
\wt{A}(i, j)=\sum_{p\le i, q\le j}a_{pq}\]
for all $i, j$. 
Define \emph{ASM order} $A\le B$ on $\A_n$ if $\wt{A}(i, j)\ge \wt{B}(i, j)$ for all $i, j$.
}\ed

\begin{figure}
\caption{$(\mathcal{A}_3, \le)$ with 4 bigrassmannian permutations} 
\label{a3}
\[
\xymatrix@C=13mm{
&*+[F]{\left(\begin{array}{ccc}0  &0   &1   \\ 0 & 1  & 0  \\1  & 0  &0  \end{array}\right)}
\ar@{-}[dl]\ar@{-}[dr]
&\\
*{
\fboxrule2pt\fcolorbox[gray]{0}{0.95}{
$\left(\begin{array}{ccc}0 &1   &0   \\0  &0   &1   \\ 1 & 0  &0  \end{array}\right)$
}
}
\ar@{-}[dr]
&&*
{
\fboxrule2pt\fcolorbox[gray]{0}{0.95}
{
$\left(\begin{array}{ccc}0 &0   &1   \\ 1 &0  &0   \\0  & 1  &0  \end{array}\right)$
}
}\ar@{-}[dl]
\\
&*+[F]{\left(\begin{array}{ccc} 0 &1   &0   \\1  &-1   &1   \\0 &1   &0  \end{array}\right)}
\ar@{-}[dl]\ar@{-}[dr]
\\
*{
\fboxrule2pt\fcolorbox[gray]{0}{0.95}{
$\left(\begin{array}{ccc} 1 & 0  &0   \\ 0 &0  &1   \\0  &1   &0  \end{array}\right)$
}
}
\ar@{-}[dr]
&&*{
\fboxrule2pt\fcolorbox[gray]{0}{0.95}{
$\left(\begin{array}{ccc}  0& 1  &0  \\ 1 & 0  &0   \\0  &0   &1  \end{array}\right)$
}
}
\ar@{-}[dl]
\\
&*+[F]{\left(\begin{array}{ccc}  1& 0  &0   \\0  &1   &0   \\ 0 &0   &1  \end{array}\right)}
&
}\]
\end{figure}
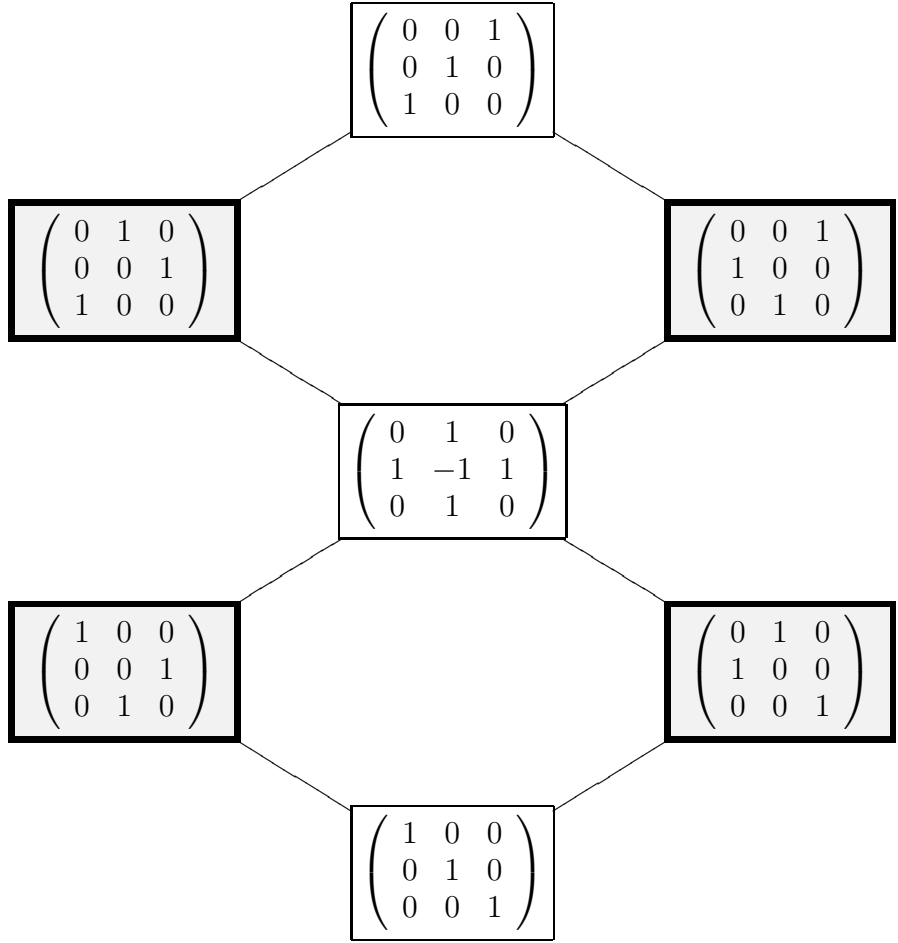

This is the traditional way to introduce a partial order onto ASMs; for example, Figure \ref{a3} shows this order on seven ASMs in $\A_3$.

Recall from the poset theory that 
$y$ \eh{covers} $x$ in a poset (write $x\lhd y$) if 
$x\lneq y$ and $x\le z\le y$ implies $z\in\{x, y\}$.

\begin{defn}
Say $A\in \A_{n}$ is \eh{join-irreducible} if it covers exactly one element in $(\A_{n}, \le)$. 
\end{defn}

It is the fact that $A$ is bigrassmannian 
if and only if it is join-irreducible in $(\A_{n}, \le)$. 
For such details on bigrassmannian permutations, see Geck-Kim \cite{geck}, Lascoux-Sch\"{u}tzenberger \cite{ls} and Reading \cite{reading}; more recent is 
Engbers-Hammet \cite{engbers}.

 Note that $(\A_{n}, \le)$ is a poset containing 
$(S_{n}, \le)$ as a subposet; 
to be more precise, 
$(\A_n, \le)$ is a finite distributive lattice 
as the MacNeille completion of Bruhat order (the smallest lattice which contains $(S_{n}, \le)$), 
the poset of all order ideals on the set of  bigrassmannian permutations as Lascoux-Sch\"{u}tzenberger \cite{ls} showed.
Recall from the lattice theory that every finite distributive $L$ is graded with the rank function
\[
r(x)=|\{j\in L\mid j\le x, \,\te{$j$ is join-irreducible}\}| \text{ for} \,\, x\in L.
\]
Now, what can we say about this rank function 
for $(\A_n, \le)$ as the MacNeille completion of the graded poset $(S_{n}, \le, I)$? This motivates us to study the following function:
\begin{defn}
\[
\beta(B)=|\{A\in \A_{n}\mid A\le B \text{ and } A \te{ is bigrassmannian (join-irreducible)}\}|.
\]
\end{defn}
It is thus natural to ask if there exists a similar formula in terms of inversions (the bottom-right part of Table \ref{t1}). The answer is yes. 
Under the identification $w\leftrightarrow (w_{ij})_{i, j=1}^{n}$, we can rewrite the formulas in Table \ref{t1} as \begin{align*}
	I(w)&=\dsum_{i<j, w(i)>w(j)}w_{iw(i)}w_{jw(j)}=\sum_{i<j, k<l}w_{jk}w_{il},
	\\\bw&=\dsum_{i<j, w(i)>w(j)}(w(i)-w(j))=
	\sum_{i<j, k<l}(l-k)w_{jk}w_{il}.
\end{align*}
\begin{thm}
\label{mth0}
For each ASM $A$, $\beta(A)$ is equal to 
the total sum of the weight of its inversions:
\[
\beta(A)=\dsum_{i<j, k<l}(l-k)a_{jk}a_{il}.
\]
\end{thm}
After the proof, we will see several consequences in Section \ref{rmks} with some open problems and more ideas; again, Bruhat order and ASMs have been of great importance as application of lattice, group and matrix theory.
With our results, we wish to contribute to development of 
further research in those areas.

\begin{rmk}\label{this}
This formula includes $O(n^{4})$ terms while 
Brualdi-Schroeder \cite[Theorem 3.1]{bs} implicitly gives the alternate formula for $\beta$ (which they called Bruhat-rank) with $O(n^{2})$ terms:
\[
\b(A)=\sum_{i, j=1}^{n}(\delta_{ij}-a_{ij})(n-i+1)(n-j+1).
\]
\end{rmk}

\subsection{Summary of this paper}
In Section 2, we recall the idea of inversions for ASMs.
In Section 3, we give details of covering relations of 
ASMs which will be the key idea to prove Theorem \ref{mth0}.
In Section 4, we give its complete proof. 
In Section 5, we discuss some consequences and open problems for our future research.


\section{Inversions for alternating sign matrices}



As mentioned in Introduction, say $(i, j)$ is an inversion of $w\in S_{n}$
if $i<j$ and $w(i)>w(j)$. 
Let us generalize this idea for ASMs.
\begin{defn}
We say that $(i, j, k, l)$ is an \eh{inversion} of $A$ if 
$i<j, k<l$ and $a_{jk}a_{il}\ne 0$. 
Also, let us say that $l-k$ is the \eh{weight} of this inversion.
Define the \emph{inversion number} of $A$ by
\[
I(A)=\sum_{
\substack{i<j, k<l,\\
a_{jk}a_{il}\ne0}
}a_{jk}a_{il}.
\]
\end{defn}
\begin{rmk}
Note that it makes no difference to include zero terms into the sum:
\[
I(A)=\sum_{
{\substack i<j, k<l}
}a_{jk}a_{il}.
\]
Sometimes this expression is more convenient. 

Positions of such two entries look like this:
\[
\left(\begin{array}{cccccc}
		&   &   &  \\
		&   &   &  a_{il}\\
		&   &    \reflectbox{$\ddots$} &&\\
		&   a_{jk}&   &  \\
		&   &   &  \\
\end{array}\right)		
\]
\end{rmk}

\begin{ex}
Let 
$A=\left(\begin{array}{cccc}
0&0&1&0\mst\\
0&1&-1&1\mst\\
1&-1&1&0\mst\\
0&1&0&0\mst\\
\end{array}\right)$.
Then, observe that 
\begin{align*}
	I(A)&=a_{31}(a_{22}+a_{13}+a_{23}+a_{24})
+a_{22}a_{13}+a_{32}(a_{13}+a_{23}+a_{24})
	\\&\ph{=}+a_{42}(a_{13}+a_{23}+a_{33}+a_{24})
+a_{33}a_{24}
	\\&=2+1-1+2+1=5.
\end{align*}
\end{ex}

\begin{notn}
For convenience, we will write $I(A, B)=I(B)-I(A)$ and $\beta(A, B)=\beta(B)-\beta(A)$ etc. whenever $A\le B$.
\end{notn}

%


\section{Covering relations of ASMs}

\newcommand{\half}{\ensuremath{\ff{1}{\,2\,}}}
\newcommand{\nhalf}{\ensuremath{-\ff{1}{\,2\,}}}

We now describe details of the order structure of ASMs; roughly speaking, starting with the unit matrix, every ASM can be constructed  by ``locally exchanging" exactly one consecutive minor of size 2. Each exchange edge is labeled by \eh{exchanging positions} and \eh{type} 1-16 according to those eight entries. 
Although this sequential construction of ASMs 
essentially follows from the earlier work of 
Lascoux-Sch\"{u}tzenberger \cite{ls} in 1996, 
it was only recently explained in Brualdi-Kiernan-Meyer-Schroeder \cite[Theorem 6.3]{bkms} and the author \cite[Proposition 3.15]{kob2} in 2017 independently. 

First, for convenience, let $1\le r, s\le n-1$ and 
\[
R=\{(r, s), (r, s+1), (r+1, s), (r+1, s+1)\}
\]
be four positions in a matrix.
\begin{lem}
\label{answer}
Let $A, B\in \A_n$.
Then, $A\lhd B$ if and only if 
there exists a unique position $(r, s)$ 
(\eh{exchange positions}) such that 
the entries at $(r, s), (r, s+1), (r+1, s), (r+1, s+1)$ positions of $A, B$
satisfy 
\[
\left(\begin{array}{cc}b_{rs}& b_{r, s+1} \\b_{r+1, s}  &b_{r+1, s+1}  \end{array}\right)-
\left(\begin{array}{cc}a_{rs}& a_{r, s+1} \\a_{r+1, s}  &a_{r+1, s+1}  \end{array}\right)
=
\left(\begin{array}{cc}-1& 1 \\1  &-1  \end{array}\right)\]
 (as listed in Table \ref{d2}) and moreover, 
 $a_{pq}=b_{pq}$ whenever $(p, q)\not\in R$. In this case, necessarily $\beta(A, B)=1$.
\end{lem}
\begin{lem}
Let $I_{n}$ be the unit matrix of size $n$.
Then, for every $A\in \A_{n}$, there is a 
sequence of ASMs $A_{0}, A_{1}, \ds, A_{\beta(A)}$ such that 
\[
I_{n}=A_{0}\lhd A_{1}\lhd A_{2}\lhd \cd \lhd A_{\beta(A)}=A.
\]
In particular, $(\A_{n}, \le, \b)$ is a graded poset with 
$\beta(I_{n})=0$.
\label{answer2}
\end{lem}

\begin{proof}[Proof of Lemmas \ref{answer} and \ref{answer2}]
These are consequences of {\cite[Key Lemma 3.18 and Proposition 3.23]{kob2}}.
\end{proof}

Figure \ref{f2} shows all covering relations of $\A_{4}$;
\fboxrule2pt\fcolorbox[gray]{0}{0.85}{\ph{1243}}
indicates the 10 bigrassmannian permutations.

\begin{ex}
Let
$A=\left(\begin{array}{cccc}
0&0&1&0\mst\\
0&1&-1&1\mst\\
1&-1&1&0\mst\\
0&1&0&0\mst\\
\end{array}\right)$. 
There exist precisely 7 bigrassmannian (join-irreducible) permutations weakly below $A$:
\[
\{1342, 1423, 3124, 2314, 1243, 1324, 2134\}.
\]
So $\beta(A)=7.$
Next, let $B=3412$. This is bigrassmannian itself and 
\[
A=
\left(\begin{array}{cccc}
0&0&1&0\mst\\
0&\ul{\,1\mst\,}&\ul{\mst -1}&1\mst\\
1&\ul{-1\mst }&\ul{\,\mst 1\,}&0\mst\\
0&1&0&0\mst\\
\end{array}\right)\lhd 
\left(\begin{array}{cccc}
0&0&1&0\mst\\
0&\ul{\,\mst 0\,}&\ul{\,\mst 0\,}&1\mst\\
1&\ul{\,\mst 0\,}&\ul{\,\mst 0\,}&0\mst\\
0&1&0&0\mst\\
\end{array}\right)=B\]
is a type 4 covering relation so that $\beta(B)=8$.
\end{ex}

\section{Main theorem}
\renewcommand{\b}{\beta}
\fboxrule0.75pt

\begin{thm1}
\label{mth1}
For each ASM $B=(b_{ij})\in \A_{n}$, $\beta(B)$, the number of bigrassmannian permutations which are weakly below $B$ under ASM order, is equal to 
the total sum of the weight of its inversions:
\[
\b(B)=\sum_{i<j, k<l}(l-k)b_{jk}b_{il}.
\]
\end{thm1}


%

%

\begin{proof}
Let 
\[
J(B):=\dsum_{i<j, k<l}(l-k)b_{jk}b_{il}.
\]
Since $I_{n}$ has no inversion, we have $J(I_{n})=0\,\, (=\beta(I_{n}))$. 
Thanks to Lemma \ref{answer2}, it is enough to show that if $A\lhd B$, then $J(B)=J(A)+1$;
 then $\beta$ and $J$ satisfy exactly same recurrences so that 
$\beta(B)=J(B)$ for all $B$.

Now assume $A\lhd B$. As in Lemma \ref{answer}, say $(r, s)$ is the position such that 
\[
\left(\begin{array}{cc}b_{rs}& b_{r, s+1} \\b_{r+1, s}  &b_{r+1, s+1}  \end{array}\right)-
\left(\begin{array}{cc}a_{rs}& a_{r, s+1} \\a_{r+1, s}  &a_{r+1, s+1}  \end{array}\right)
=
\left(\begin{array}{cc}-1& 1 \\1  &-1  \end{array}\right)\]
and 
\[
a_{pq}=b_{pq} \te{ if } (p, q)\not\in R
\]
where $R=\{(r, s), (r, s+1), (r+1, s), (r+1, s+1)\}$.
We will carefully compute 
\[
J:=J(A, B)=\sum_{i<j, k<l}(b_{jk}b_{il}-a_{jk}a_{il})(l-k)
\]
with dividing it into 8 terms
\[
J=J_{0}+J_{1}+J_{2}+J_{3}+J_{4}+J_{5}+J_{6}+J_{7}
\]
as each $J_{i}$ introduced below
(We will see that $J_{0}=J_{2}=J_{3}=J_{5}=J_{6}=0$).

\begin{enumerate}
	\item[{(Case 0)}]
	If $(j, k), (i, l)\not\in R$, then 
$a_{il}=b_{il}, a_{jk}=b_{jk}$
so that 
\[
J_{0}:=\sum_{\substack{i<j, k<l\\
(j, k), (i, l)\not\in R}}
\underbrace{(b_{jk}b_{il}-a_{jk}a_{il})}_{0}(l-k)=0.
\]

	\item[{(Case 1)}] $j=r+1, 1\le k\le s-1$.
\newcommand{\Fb}[1]{{\fboxsep5pt{\fb{#1}}}}
\[
{\renewcommand{\arraystretch}{1.5}
\left(\begin{array}{ccc|cc}
		&   &   &  \\
		&   &a_{rs}&a_{r, s+1}   &  \\\hl
\st\fb{$a_{r+1, k}$}&\cd	&   a_{r+1, s}&a_{r+1, s+1}   &  \\
		&   &   &  \\
\end{array}\right)}
\]
\begin{align*}
	J_{1}&:=\sum_{k=1}^{s-1}
\k{
(b_{r+1, k}b_{rs}-a_{r+1, k}a_{rs})(s-k)
+
(b_{r+1, k}b_{r,s+1}-a_{r+1, k}a_{r, s+1})(s+1-k)
}
	\\&=\sum_{k=1}^{s-1}
\k{b_{r+1, k}\underbrace{(b_{rs}-a_{rs})}_{-1}(s-k)
+b_{r+1, k}\underbrace{(b_{r,s+1}-a_{r,s+1})}_{1}(s+1-k)}
	\\&=\sum_{k=1}^{s-1}b_{r+1, k}.
\end{align*}
	\item[{(Case 2)}]
 $r+2\le j\le n, 1\le k\le s-1$.
\[
{\renewcommand{\arraystretch}{1.5}
\left(\begin{array}{ccc|cc}
		&   &   &  \\
	&	&   a_{rs}&a_{r, s+1}   &  \\\hl
	&&   a_{r+1, s}&a_{r+1, s+1}   &  \\
		&
		  \reflectbox{$\ddots$} &   &  \\
\fb{$a_{jk}$}		&   &   &  \\
\end{array}\right)}
\]
\begin{align*}
	J_{2}&:=
	\sum_{
	{\substack {r+2\le j\le n\\1\le k\le s-1}}}
b_{jk}(b_{rs}(s-k)+b_{r, s+1}(s+1-k)+b_{r+1, s}(s-k)+b_{r+1, s+1}(s+1-k))
	\\&\ph{=}
	-\sum_{
	{\substack {r+2\le j\le n\\1\le k\le s-1}}}
	a_{jk}(a_{rs}(s-k)+a_{r, s+1}(s+1-k)+a_{r+1, s}(s-k)+
	a_{r+1, s+1}(s+1-k))
	\\&=\sum_{
	{\substack {r+2\le j\le n\\1\le k\le s-1}}}b_{jk}
	\k{-1(s-k)+1(s+1-k)+1(s-k)-1(s+1-k)}
	\\&=0.
\end{align*}
	\item[{(Case 3)}]
 $r+2\le j\le n, k=s$.
	\[
{\renewcommand{\arraystretch}{1.5}
\left(\begin{array}{ccc|cc}
		&   &   &  \\
		&   &a_{rs}&a_{r, s+1}   &  \\\hl
	&	&   a_{r+1, s}&a_{r+1, s+1}   &  \\
		&   &\vdots&   &  \\
		&   &\fb{$a_{js}$}&   &  \\
\end{array}\right)}
\]
\begin{align*}
	J_{3}&:=\sum_{j=r+2}^{n}
	\k{b_{js}(b_{r, s+1}+b_{r+1, s+1})
	-a_{js}(a_{r, s+1}+a_{r+1, s+1})}
	\\&	=\sum_{j=r+2}^{n}
	b_{js}(
	\underbrace{(b_{r, s+1}-a_{r, s+1})}_{1}
	+\underbrace{(b_{r+1, s+1}-a_{r+1, s+1})}_{-1}
	)
	\\&=0.
\end{align*}
By symmetry of rows and columns, we will get similar results in (Case 4) -- (Case 6):
	\item[{(Case 4)}]
 $i=r, s+2\le l\le n$.
\[
{\renewcommand{\arraystretch}{1.5}
\left(\begin{array}{ccc|ccc}
		&   &   &  \\
		&   &a_{rs}&a_{r, s+1}   & \dots&\fb{$a_{rl}$}\st
		\dep{10pt} \\\hl
	&	&   a_{r+1, s}&a_{r+1, s+1}   &  \\
\end{array}\right)}
\]
\begin{align*}
	J_{4}&:=
	\sum_{l=s+2}^{n}
	\k{b_{r+1, s}b_{rl}(l-s)+b_{r+1, s+1}b_{rl}(l-(s+1))}
	\\&\ph{=}
	-\sum_{l=s+2}^{n}
	\k{a_{r+1, s}a_{rl}(l-s)+a_{r+1, s+1}a_{rl}(l-(s+1))}
	\\&=\sum_{l=s+2}^{n}((l-s)-(l-s-1))b_{rl}
	=\sum_{l=s+2}^{n}b_{rl}.
\end{align*}	
	
	\item[{(Case 5)}]
 $1\le i\le r-1, s+2\le l\le n$.
	\[
{\renewcommand{\arraystretch}{1.5}
\left(\begin{array}{ccc|ccc}
		&   &   &&&  \fb{$a_{il}$}\\
		&   & &&\reflectbox{$\ddots$}  &  \\
		&   &a_{rs}&a_{r, s+1}   &  \\\hl
	&	&   a_{r+1, s}&a_{r+1, s+1}   &  \\
\end{array}\right)}
\]
\begin{align*}
	J_{5}&:=
	\sum_{
	{\substack {1\le i\le r-1\\s+2\le l\le n}}
	}
	(
	(b_{rs}-a_{rs})(l-s)+
	(b_{r,s+1}-a_{r,s+1})(l-s-1)
	\\&\ph{\sum_{
	{\substack {1\le i\le r-1\\s+2\le l\le n}}
	}}
	+
	(b_{r+1,s}-a_{r+1,s})(l-s)+
	(b_{r+1,s+1}-a_{r+1,s+1})(l-s-1))
	b_{il}
	\\&=0.
\end{align*}
	\item[{(Case 6)}]
 $1\le i\le r-1, l=s+1$.
	\[
{\renewcommand{\arraystretch}{1.5}
\left(\begin{array}{ccc|ccc}
		&  & &\fb{$a_{i, s+1}$}   &  \\
		&   &&  \vdots &  \\
		&   &a_{rs}&a_{r, s+1}   &  \\\hl
	&	&   a_{r+1, s}&a_{r+1, s+1}   &  \\
\end{array}\right)}
\]
\begin{align*}
	J_{6}&:=\sum_{i=1}^{r-1}
	\k{(b_{rs}-a_{rs})(s+1-s)+(b_{r+1, s}-a_{r+1, s})(s+1-s)}b_{i, s+1}
	\\&=0.
\end{align*}
So far, we have
\[
J=J(A, B)=J_{0}+\cd+J_{7}=J_{1}+J_{4}+J_{7}.
\]
	\item[{(Case 7)}]
 Finally, let $J_{7}:=b_{r+1, s}b_{r, s+1}-a_{r+1, s}a_{r, s+1}.$ 
	This is 1, 0, or $-1$.
\end{enumerate}
\begin{itemize}
	\item Type 1, 5, 9, 13 in Table \ref{d2} ($J_{7}=1$):
$b_{r+1, s}=b_{r, s+1}=1$. 
Due to the property on a partial sum of row entries of ASMs, 
$b_{r+1, s}=1$ implies 
\begin{align*}
	1&\geq\sum_{k=1}^{s}b_{r+1, k}=
\underbrace{\sum_{k=1}^{s-1}b_{r+1, k}}_{J_{1}}+\underbrace{b_{r+1, s}}_{1},
	\\J_{1}&=0,
	\\J_{4}&=\sum_{l=s+2}^{n}b_{rl}=0 \q \te{(likewise)\, and} 
	\\J&=0+0+1=1.
\end{align*}
\item Type 2, 6, 10, 14 ($J_{7}=0$): $b_{r+1, s}=1$ implies $J_{1}= \dsum_{k=1}^{s-1}b_{r+1, k} =0$
and similarly $a_{r, s+1}=-1$ implies 
$J_{4}=\dsum_{l=s+2}^{n}b_{rl}=\dsum_{l=s+2}^{n}a_{rl}=1$.
Thus, 
\[
J=0+1+0=1.\]
\item Type 3, 7, 11, 15 ($ J_{7}=0$): 
Likewise, 
$J_{1}=1$ and $J_{4}=0$ so that 
\[
J=1+0+0=1.
\]
\item Type 4, 8, 12, 16 ($J_{7}=-1$): we have 
$a_{r+1, s}=a_{r, s+1}=-1$ so that 
\begin{align*}
	J_{1}&=\sum_{k=1}^{s-1}b_{r+1, k}
=\sum_{k=1}^{s-1}a_{r+1, k}=1,
	\\J_{4}&=\sum_{l=s+2}^{n}b_{rl}
=\sum_{l=s+2}^{n}a_{rl}=1,
	\\J&=1+1+(-1)=1.
\end{align*}
\end{itemize}
Conclude that $J=1$ in any case.
\end{proof}

\begin{ex}
Let
$A=\left(\begin{array}{cccc}
0&0&1&0\mst\\
0&1&-1&1\mst\\
1&-1&1&0\mst\\
0&1&0&0\mst\\
\end{array}\right)$ and $B=3412$. Then, observe that 
\begin{align*}
	J(A)&=1a_{31}a_{22}+2a_{31}a_{13}+2a_{31}a_{23}+3a_{31}a_{24}
+1a_{22}a_{13}
\\&\ph{=}+1a_{32}a_{13}+1a_{32}a_{23}+2a_{32}a_{24}
	\\&\ph{=}+1a_{42}a_{13}+1a_{42}a_{23}+1a_{42}a_{33}+2a_{42}a_{24}
+1a_{33}a_{24}
	\\&=1+2-2+3+1-1+1-2+1-1+1+2+1=7=\beta(A) \te{ and }
	\\J(B)&=(3-1)+(3-2)+(4-1)+(4-2)=8=\beta(A)+1.
\end{align*}
%
%
%
\end{ex}

\begin{cor}
\[
\beta(A)=\sum_{i<j, k<l}(j-i)a_{jk}a_{il}.
\]\end{cor}

\begin{proof}
Considering corner sum matrices, 
the map $A\mto A^{T}$ (transpose) is an order-preserving isomorphism on $(\A_{n}, \le, \beta)$ which therefore must preserve $\beta$. 
Writing the $(p, q)$-entry of $A^{T}$ as $a'_{pq}$ (of course $a'_{pq}=a_{qp}$), we have 
\begin{align*}
	\beta(A)&=\beta(A^{T})
	\\&=\sum_{i<j, k<l}(l-k)a_{jk}'a'_{il}
	=\sum_{i<j, k<l}(l-k)a_{kj}a_{li}
	\\&=\sum_{k<l, i<j}(j-i)a_{il}a_{jk}\q \mb{($i\mto k, j\mto l, k\mto i, l\mto j$)}
	\\&=\sum_{i<j, k<l}(j-i)a_{jk}a_{il}.
\end{align*}
\end{proof}


\section{Final Remarks}
\label{rmks}

At the end, we discuss several consequences of Theorem \ref{mth0}.

\subsection{Inversions and $\b$}

We have seen interactions of $(S_{n}, \le, I)$ and $(\A_{n}, \le, \beta)$. In particular, $I$ and $\beta$ are the rank functions of these graded posets and hence both increasing. What we should not misunderstand, however, is that inversion numbers are \eh{not} necessarily increasing along an edge $A\lhd B$.
For example, with 
\[
A=
\left(\begin{array}{cccc}
0&0&1&0\mst\\
0&1&-1&1\mst\\
1&-1&1&0\mst\\
0&1&0&0\mst\\
\end{array}\right) \te{  \q and \q } 
B=\left(\begin{array}{cccc}
0&0&1&0\mst\\
0&0&0&1\mst\\
1&0&0&0\mst\\
0&1&0&0\mst\\
\end{array}\right)=3412,\]
we have $A\lhd B$ but $I(A)=5>4=I(B)$. At least we can say this:
\begin{cor}
If $A\lhd B$, then $I(A, B)\in \{-1, 0, 1\}$. 
As a result, $I(A)\le \beta(A)$ for all $A$.
\end{cor}
\begin{proof}
This proof is quite similar to the one given in 
Theorem \ref{mth1} without all factors arising from weights of inversions 
(if we compute $I=I(A, B)=I_{0}+\cd+I_{7}$ in the same way,
then we will have $I_{0}=\cd=I_{6}=0$ and $I_{7}\in \{-1, 0, 1\}$).
\end{proof}

\subsection{Non-symmetry of inversions}

It is well-known that 
$\dsum_{w\in S_{n}}\lam^{I(w)}=\di\prod_{k=1}^{n}(1+\lam +\cd+\lam ^{k-1})$. 
However, it seems very difficult to 
find $\dsum_{A\in \A_{n}} \lam^{I(A)}$.
 We could not find any answer in the litearture.
\begin{open}
Can we give a polynomial-time algorithm to compute it?
\end{open}
Below, we wish to explain a detail of this problem 
on inversions and will propose perhaps an easier problem.

\begin{defn}
Let $(P, \le)$ be a poset. Its \eh{dual} is the poset $(P^{*}, \le^{*})$ with $P^{*}=P$ as sets, and $x\le y\iff y\le^{*}x$ for $x, y\in P$.
Say $(P, \le)$ is \eh{self-dual} if there exists a bijection $f:P\to P^{*}$ such that $x\le y \iff f(x)\le^{*}f(y)$.
\end{defn}
As seen below, $(S_{n}, \le, I)$ and $(\A_{n}, \le, \beta)$ are both self-dual graded posets so that these rank generating functions
$\dsum_{w\in S_{n}}\lam^{I(w)}, 
\dsum_{A\in \A_{n}}q^{\beta(w)}$ must be symmetric (here, symmetric means coefficients are palindromic); 

\renewcommand{\l}{\lam}
However, inversions on ASMs are not like this:
for example,
\[
\sum_{A\in \A_{3}} \lam^{I(A)} =
(1+\lam)(1+\lam+\lam^{2})+\lam^{2}
=1+2\l+3\l^{2}+\l^{3}
\]
is clearly not symmetric. Apparently, the middle ASM
$A=\left(\begin{array}{ccc} 0 &1   &0   \\1  &-1   &1   \\0 &1   &0  \end{array}\right)$ with $I(A)=2$ makes this happen.
But, with little modification of $I(A)$, we can construct some function 
which is more symmetric. For this purpose, the following ideas are useful:
\begin{defn}
Let $w_{0}$ denote the reverse permutation: $w_{0}(i)=n-i+1$. 
For $A\in \A_{n}$, the \eh{dual} of $A$ is $A^{*}=w_{0}A$ (the matrix reading rows of $A$ backwards).
\end{defn}
\begin{defn}
Say $(i, j, k, l)$ is a \eh{dual inversion of} an ASM $A$ if 
$i<j, k<l$ and $a_{ik}a_{jl}\ne 0$. The \eh{dual inversion number} of $A$ 
is 
\[
I^{*}(A)=\sum_{i<j, k<l}a_{ik}a_{jl}.
\]\end{defn}
In fact, $I^{*}(A)=I(A^{*})$.
%
%

\begin{lem}\label{lem2}
$A\le B\iff B^{*}\le A^{*}$ in $\A_{n}$.
\end{lem}

\begin{proof}
This follows from Lemma \ref{answer}.
\end{proof}

\begin{defn}
Let $N(A)$ be the number of $-1$ in entries of $A$. 
\end{defn}
Note that $N(A)=N(A^{*})$.

\begin{thm}\label{in}
For all $B\in \A_{n}$, we have 
\[
I(B)+I^{*}(B)-N(B)=\ff{1}{\,2\,}n(n-1).
\]
Here we used the letter $B$ to directly apply Lemma \ref{answer} in the proof below.
\end{thm}

For the proof, it is convenient to introduce the following.
\begin{defn}
Define $H(B)=I(B)-N(B)/2$. Call this the \eh{weak inversion number} of $B$.
\end{defn}
Then, the statement of Theorem \ref{in} is equivalent to
\[
H(B)+H(B^{*})=\ff{1}{\,2\,}n(n-1)
\]
for all $B\in \A_{n}$.

\begin{proof}[Proof of Theorem \ref{in}]
Define the function by $F(B)=H(B)+H(B^{*})$. 
We prove $F(B)=n(n-1)/2$ for all $B$ by induction on $\beta(B)$.
Suppose $\beta(B)\ge 1$. 
Choose $A$ such that $A\lhd B$ (Lemma \ref{answer}). 
It follows from Lemma \ref{lem2} that $B^{*}\lhd A^{*}$.
As easily seen from Table \ref{d2}, $H(A, B)$ of type $m$ and $H(B^{*}, A^{*})$ of type $m^{*}$ are equal in all cases. Thus
\[
F(A, B)=H(A, B)-H(B^{*}, A^{*})=0
\]
and therefore $F(B)$ is constant. Conclude that $F(B)=n(n-1)/2$.
\end{proof}

\begin{fact}
$\max \{H(A)\mid A\in \A_{n}\}=n(n-1)/2$. 
Moreover, 
\[
H(A)=\ff{1}{\,2\,}n(n-1) \iff A=w_{0}.\]
\end{fact}
\begin{proof}
We always have $H(B)\ge 0$ since 
an ASM pattern 
$\left(\begin{smallmatrix}
	&+   &   \\
	+&-   &+   \\
	&+   &
\end{smallmatrix}\right)$ 
forces there to be at least two inversions per minus sign, i.e., $2I(B)\ge N(B)$. 
Consequently, $H(A)\le H(A)+H(A^{*})=n(n-1)/2$ and in particular $H(w_{0})=I(w_{0})-N(w_{0})/2=n(n-1)/2$. 
Now suppose $H(A)=n(n-1)/2$. Then  $H(A^{*})=0$ thanks to Theorem \ref{in}.
Further, the equality $H(A^{*})=0$ forces $N(A^{*})=0$, $I(A^{*})=0$, $A^{*}=I_{n}$ and hence $A=w_{0}$.
\end{proof}


\begin{cor}
$\dsum_{A\in \A_{n}}\lam^{H(A)}$ is monic and symmetric with the highest degree term $\l^{n(n-1)/2}$.
\end{cor}
It is easy to see that 
\[
\dsum_{A\in \A_{3}}\lam^{H(A)}=1+2\l+\l^{3/2}+2\l^{2}+\l^{3}.
\]
Next, let us compute $\dsum_{A\in \A_{4}}\lam^{H(A)}$.
\[
\begin{smat}
1&0&0&0\mst\\
0&0&1&0\mst\\
0&1&-1&1\mst\\
0&0&1&0\mst\\
\end{smat},
\begin{smat}
0&1&0&0\mst\\
1&-1&1&0\mst\\
0&1&0&0\mst\\
0&0&0&1\mst\\
\end{smat}:
I=2, N=1, H=3/2,
\]
\[
\begin{smat}
0&1&0&0\mst\\
1&-1&1&0\mst\\
0&1&-1&1\mst\\
0&0&1&0\mst\\
\end{smat}:
I=3, N=2, H=2,
\]
\[
\begin{smat}
0&1&0&0\mst\\
1&-1&1&0\mst\\
0&0&0&1\mst\\
0&1&0&0\mst\\
\end{smat},
\begin{smat}
0&1&0&0\mst\\
1&-1&0&1\mst\\
0&1&0&0\mst\\
0&0&1&0\mst\\
\end{smat},
\begin{smat}
0&0&1&0\mst\\
1&0&0&0\mst\\
0&1&-1&1\mst\\
0&0&1&0\mst\\
\end{smat},
\begin{smat}
0&1&0&0\mst\\
0&0&1&0\mst\\
1&0&-1&1\mst\\
0&0&1&0\mst\\
\end{smat},
\begin{smat}
0&0&1&0\mst\\
1&0&-1&1\mst\\
0&1&0&0\mst\\
0&0&1&0\mst\\
\end{smat}:
I=3, N=1, H=5/2,
\]
\[
\begin{smat}
0&1&0&0\mst\\
1&-1&0&1\mst\\
0&0&1&0\mst\\
0&1&0&0\mst\\
\end{smat}
:I=4, N=1, H=7/2.
\]
All other 9 non-permutations are the dual of these. Thus, 
\begin{align*}
\dsum_{A\in \A_{4}}\lam^{H(A)}&=
(1+\l)(1+\l+\l^{2})(1+\l+\l^{2}+\l^{3})
\\&\ph{=}+(
2\l^{3/2}+\l^{2}+5\l^{5/2}+\l^{7/2}
+
2\l^{9/2}+\l^{4}+5\l^{7/2}+\l^{5/2}
)
\\&=
	1+3\l+2\l^{3/2}+6\l^{2}+6\l^{5/2}+6\l^{3}+6\l^{7/2}
+6\l^{4}+2\l^{9/2}+3\l^{5}+\l^{6}
\end{align*}
is symmetric.
\begin{cor}
If $A\lhd B$, then 
\[
H(A, B)\in \left\{-1, -\ff{1}{\,2\,}, 0, \ff{1}{\,2\,}, 1\right\}.
\]
\end{cor}
\begin{proof}
See Table \ref{d2}.
\end{proof}

\begin{open}
What is $\dsum_{A\in \A_{n}}\lam^{H(A)}$ (a polynomial in $\l^{1/2}$)? 
\end{open}

\begin{rmk}
We can regard $H(A)$ as a sort of weighted counting of inversions:
For each $(p, q)$, let 
\[
H_{pq}(A)=a_{pq}
\k{
\ff{1}{\,2\,}
\k{\sum_{r>p, s<q}a_{rs}
+\sum_{r<p, s>q}a_{rs}
}
+\ff{1}{\,4\,}
\k{
\sum_{r=1}^{p-1}a_{rq}
+
\sum_{s=q+1}^{n}a_{ps}
}
}.
\]
If $a_{pq}=0$, then, $H_{pq}(A)=0$.
If $a_{pq}=1$, then 
\begin{align*}
	H_{pq}(A)&=a_{pq}
\ff{1}{\,2\,}
\k{\sum_{r>p, s<q}a_{rs}
+\sum_{r<p, s>q}a_{rs}
}
+\ff{1}{\,4\,}
\underbrace{\k{
\sum_{r=1}^{p-1}a_{rq}
+
\sum_{s=q+1}^{n}a_{ps}
}}_{0+0}
	\\&=\ff{1}{\,2\,}
\k{\sum_{r>p, s<q}a_{pq}a_{rs}
+\sum_{r<p, s>q}a_{pq}a_{rs}
}
.
\end{align*}
If $a_{pq}=-1$, then 
\begin{align*}
	H_{pq}(A)&=a_{pq}
\ff{1}{\,2\,}
\k{\sum_{r>p, s<q}a_{rs}
+\sum_{r<p, s>q}a_{rs}
}
-\ff{1}{\,4\,}
\underbrace{
\k{
\sum_{r=1}^{p-1}a_{rq}
+
\sum_{s=q+1}^{n}
a_{ps}}
}_{1+1}
	\\&=\ff{1}{\,2\,}
\k{\sum_{r>p, s<q}a_{pq}a_{rs}
+\sum_{r<p, s>q}a_{pq}a_{rs}
}
-\ff{1}{\,2\,}.
\end{align*}
Altogether, we get
\[
H(A)=\sum_{p, q=1}^{n}H_{pq}(A)=I(A)-\ff{N(A)}{2}
\]
with $\ff{a_{jk}a_{il}}{2}$ appearing exactly twice for each inversion $(i, j, k, l)$ of $A$. This is why we called $H(A)$ the weak inversion number. 
\end{rmk}

%
%
The author recently showed \cite{kob2} that 
	\[
\dsum_{w\in S_{n}}(-1)^{I(w)}q^{\beta(w)}=
	\di\prod_{k=1}^{n-1}(1-q^{k})^{n-k}.\]
Further, we can consider bivariate statistics.
\begin{open}
Give an efficient method to compute the following generating functions:
\[
	\sum_{w\in S_{n}} \lam^{I(w)}q^{\beta(w)}, 
\q
\sum_{A\in \A_{n}} \lam^{I(A)}q^{\beta(A)}
\q 
\sum_{A\in \A_{n}} \lam^{H(A)}q^{\beta(A)}
\]
\end{open}

\subsection{Coxeter group analogy}
The symmetric group is a Coxeter group of type A; 
some other types also have a representation of 
certain permutations (matrices) and it all makes sense to speak of inversions, Bruhat order, bigrassmannian and join-irreducible elements. Try something similar in this article for type B or Affine type $\wt{\tn{A}}$ and see what happens; Geck-Kim \cite{geck} discussed some on type B 
and it seems that Reading-Waugh \cite{readingwaugh} is 
the only reference which studied Affine join-irreducible permutations.

\appendix

\begin{table}[h!]
\caption{16 types of covering relations $A\lhd B$ in $\A_{n}$}
\renewcommand{\arraystretch}{1.7}
\resizebox{0.95\tw}{!}{
\begin{tabular}{|c|c|c||c|c|c|}\hline
type&
$\left(\begin{array}{cc}b_{rs}& b_{r, s+1} \\b_{r+1, s}  &b_{r+1, s+1}  \end{array}\right)$&
$\left(\begin{array}{cc}a_{rs}& a_{r, s+1} \\a_{r+1, s}  &a_{r+1, s+1}  \end{array}\right)$
&$I(A, B)$&$\ff{1}{\,2\,}{N(A, B)}$
&$H(A, B)$
\\\hline
$1=1^{*}$&$\left(\begin{array}{cc}{0} & 1 \\1  &0  \end{array}\right)$&
$\left(\begin{array}{cc}1&0  \\0  &1  \end{array}\right)$
&1&0&1
\\\hline
$2=5^{*}$&$\left(\begin{array}{cc}{0} & 0 \\1  &0  \end{array}\right)$&$\left(\begin{array}{cc}1&  -1\\ 0 & 1 \end{array}\right)$
&0&\nhalf&\half
\\\hline
$3=9^{*}$&$\left(\begin{array}{cc}{0} & 1 \\0  &0  \end{array}\right)$&
$\left(\begin{array}{cc}1& 0 \\-1  & 1 \end{array}\right)$
&0&\nhalf&\half
\\\hline
$4=13^{*}$&$\left(\begin{array}{cc}{0} & 0 \\0  &0  \end{array}\right)$&$\left(\begin{array}{cc}1&  -1\\-1  &1  \end{array}\right)$
&$-1$&$-1$&0
\\\hline
$5=2^{*}$&$\left(\begin{array}{cc}{0} & 1 \\1  &-1  \end{array}\right)$&
$\left(\begin{array}{cc}1& 0 \\ 0 & 0 \end{array}\right)$
&1&\half&\half
\\\hline
$6=6^{*}$&$\left(\begin{array}{cc}{0} & 0 \\1  &-1  \end{array}\right)$&$\left(\begin{array}{cc}1& -1 \\0  &0  \end{array}\right)$
&0&0&0
\\\hline
$7=10^{*}$&$\left(\begin{array}{cc}{0} & 1 \\0  &-1  \end{array}\right)$&$\left(\begin{array}{cc}1& 0 \\ -1 & 0 \end{array}\right)$
&0&0&0
\\\hline
$8=14^{*}$&$\left(\begin{array}{cc}{0} & 0 \\0  &-1  \end{array}\right)$&$\left(\begin{array}{cc}1& -1 \\-1  &0  \end{array}\right)$
&$-1$&\nhalf&\nhalf
\\\hline
$9=3^{*}$&$\left(\begin{array}{cc}-1 & 1 \\1  &0  \end{array}\right)$&
$\left(\begin{array}{cc}0&0  \\0  &1  \end{array}\right)$
&1&\half&\half
\\\hline
$10=7^{*}$&$\left(\begin{array}{cc}{-1} & 0 \\1  &0  \end{array}\right)$&$\left(\begin{array}{cc}0&  -1\\ 0 & 1 \end{array}\right)$
&0&0&0
\\\hline
$11=11^{*}$&$\left(\begin{array}{cc}{-1} & 1 \\0  &0  \end{array}\right)$&
$\left(\begin{array}{cc}0& 0 \\-1  & 1 \end{array}\right)$
&0&0&0
\\\hline
$12=15^{*}$&$\left(\begin{array}{cc}{-1} & 0 \\0  &0  \end{array}\right)$&$\left(\begin{array}{cc}0&  -1\\-1  &1  \end{array}\right)$
&$-1$&\nhalf&\nhalf
\\\hline
$13=4^{*}$&$\left(\begin{array}{cc}{-1} & 1 \\1  &-1  \end{array}\right)$&
$\left(\begin{array}{cc}0& 0 \\ 0 & 0 \end{array}\right)$
&1&1&0
\\\hline
$14=8^{*}$&$\left(\begin{array}{cc}{-1} & 0 \\1  &-1  \end{array}\right)$&$\left(\begin{array}{cc}0& -1 \\0  &0  \end{array}\right)$
&0&\half&\nhalf
\\\hline
$15=12^{*}$&$\left(\begin{array}{cc}{-1} & 1 \\0  &-1  \end{array}\right)$&$\left(\begin{array}{cc}0& 0 \\ -1 & 0 \end{array}\right)$
&0&\half&\nhalf
\\\hline
$16=16^{*}$&$\left(\begin{array}{cc}{-1} & 0 \\0  &-1  \end{array}\right)$&$\left(\begin{array}{cc}0& -1 \\-1  &0  \end{array}\right)$
&$-1$&0&$-1$
\\\hline
\end{tabular}
}
\label{d2}
\end{table}%

\begin{figure}[h!]
\caption{$(\A_{4}, \lhd)$ with 10 bigrassmannian permutations}
\label{f2}
\begin{center}
\mb{}\vfill
\resizebox{1\tw}{!}
{
\begin{xy}
(0,0);<24mm,0mm>:
,(3,1.5)+(0,0)*{
\fboxrule3pt\fcolorbox[gray]{0}{0.85}{\LARGE \st \,1243\,}}="31"
,(5,1.5)+(0,0)*{\fboxrule3pt\fcolorbox[gray]{0}{0.85}{\LARGE \mst \,1324\,}}="51"
,(7,1.5)+(0,0)*{\fboxrule3pt\fcolorbox[gray]{0}{0.85}{\LARGE \mst \,2134\,}}="71"
,(9,4.5)+(0,0)*{\fboxrule3pt\fcolorbox[gray]{0}{0.85}{\LARGE \mst \,2314\,}}="93"
,(6,9)+(0,1)*{\fboxrule3pt\fcolorbox[gray]{0}{0.85}{\LARGE \mst \,4123\,}}="66"
,(5,0)*{\fboxrule1pt\fcolorbox[gray]{0}{1}
{\LARGE \mst \,1234\,}}="50"
,(0,6)*{\fboxrule1pt\fcolorbox[gray]{0}{1}
{\LARGE \mst \,1432\,}}="04"
,(1,10.5)+(0,1.0)*{\fboxrule1pt\fcolorbox[gray]{0}{1}{\LARGE \mst \,2431\,}}="17"
,(5,3)+(0,0)*{\fboxrule1pt\fcolorbox[gray]{0}{1}{\LARGE \mst \,2143\,}}="52"
,(10,6)*{\fboxrule1pt\fcolorbox[gray]{0}{1}{\LARGE \mst \,3214\,}}="a4"
,(3,10.5)+(0,1.0)*{\fboxrule1pt\fcolorbox[gray]{0}{1}{\LARGE \mst \,4132\,}}="37"
,(1,4.5)+(0,0)*{\fboxrule3pt\fcolorbox[gray]{0}{0.85}{\LARGE \mst \,1342\,}}="13"
,(2,7.5)+(0,0.5)*{\fboxrule1pt\fcolorbox[gray]{0}{1}{\LARGE \mst \,3142\,}}="25"
,(7,10.5)+(0,1.0)*{\fboxrule1pt\fcolorbox[gray]{0}{1}{\LARGE \mst \,3241\,}}="77"
,(3,4.5)+(0,0)*{\fboxrule3pt\fcolorbox[gray]{0}{0.85}{\LARGE \mst \,1423\,}}="33"
,(8,7.5)+(0,0.5)*{\fboxrule1pt\fcolorbox[gray]{0}{1}{\LARGE \mst \,2413\,}}="85"
,(9,10.5)+(0,1.0)*{\fboxrule1pt\fcolorbox[gray]{0}{1}{\LARGE \mst \,4213\,}}="97"
,(7,4.5)+(0,0)*{\fboxrule3pt\fcolorbox[gray]{0}{0.85}{\LARGE \mst \,3124\,}}="73"
,(4,9)+(0,1)*{\fboxrule3pt\fcolorbox[gray]{0}{0.85}{\LARGE \mst \,2341\,}}="46"
,(5,12)+(0,1.0)*{\fboxrule3pt\fcolorbox[gray]{0}{0.85}{\LARGE \mst \,3412\,}}="58"
,(3,13.5)+(0,1.0)*{\fboxrule1pt\fcolorbox[gray]{0}{1}{\LARGE \mst \,3421\,}}="39"
,(5,13.5)+(0,1.0)*{\fboxrule1pt\fcolorbox[gray]{0}{1}{\LARGE \mst \,4231\,}}="59"
,(7,13.5)+(0,1.0)*{\fboxrule1pt\fcolorbox[gray]{0}{1}{\LARGE \mst \,4312\,}}="79"
,(5,15)+(0,1.0)*{\fboxrule1pt\fcolorbox[gray]{0}{1}{\LARGE \mst \,4321\,}}="5a"
,(3,3)+(0,0)*++[F]{
\begin{smat}
1&0&0&0\mst\\
0&0&1&0\mst\\
0&1&-1&1\mst\\
0&0&1&0\mst\\
\end{smat}
}="32",
,(7,3)+(0,0)*++[F]
{
\begin{smat}
0&1&0&0\mst\\
1&-1&1&0\mst\\
0&1&0&0\mst\\
0&0&0&1\mst\\
\end{smat}
}="72",
,(5,4.5)+(0,0)*++[F]{
\begin{smat}
0&1&0&0\mst\\
1&-1&1&0\mst\\
0&1&-1&1\mst\\
0&0&1&0\mst\\
\end{smat}
}="53",
,(2,6)*++[F]{
\begin{smat}
0&1&0&0\mst\\
1&-1&1&0\mst\\
0&0&0&1\mst\\
0&1&0&0\mst\\
\end{smat}
}="24",
,(4,6)*++[F]{
\begin{smat}
0&1&0&0\mst\\
1&-1&0&1\mst\\
0&1&0&0\mst\\
0&0&1&0\mst\\
\end{smat}
}="44",
,(6,6)*++[F]{
\begin{smat}
0&0&1&0\mst\\
1&0&0&0\mst\\
0&1&-1&1\mst\\
0&0&1&0\mst\\
\end{smat}
}="64",
,(8,6)*++[F]{
\begin{smat}
0&1&0&0\mst\\
0&0&1&0\mst\\
1&0&-1&1\mst\\
0&0&1&0\mst\\
\end{smat}
}="84",
,(0,7.5)+(0,0.5)*++[F]{
\begin{smat}
0&1&0&0\mst\\
1&-1&0&1\mst\\
0&0&1&0\mst\\
0&1&0&0\mst\\
\end{smat}
}="05",
,(4,7.5)+(0,0.5)*++[F]{
\begin{smat}
0&1&0&0\mst\\
0&0&1&0\mst\\
1&-1&0&1\mst\\
0&1&0&0\mst\\
\end{smat}
}="45",
,(6,7.5)+(0,0.5)*++[F]{
\begin{smat}
0&0&1&0\mst\\
1&0&-1&1\mst\\
0&1&0&0\mst\\
0&0&1&0\mst\\
\end{smat}
}="65",
,(10,7.5)+(0,0.5)*++[F]{
\begin{smat}
0&0&1&0\mst\\
0&1&0&0\mst\\
1&0&-1&1\mst\\
0&0&1&0\mst\\
\end{smat}
}="a5",
,(0,9)+(0,1)*++[F]{
\begin{smat}
0&0&1&0\mst\\
1&0&-1&1\mst\\
0&0&1&0\mst\\
0&1&0&0\mst\\
\end{smat}
}="06",
,(2,9)+(0,1)*++[F]{
\begin{smat}
0&1&0&0\mst\\
0&0&0&1\mst\\
1&-1&1&0\mst\\
0&1&0&0\mst\\
\end{smat}
}="26",
,(8,9)+(0,1)*++[F]{
\begin{smat}
0&0&1&0\mst\\
0&1&0&0\mst\\
1&-1&0&1\mst\\
0&1&0&0\mst\\
\end{smat}
}="86",
,(10,9)+(0,1)*++[F]{
\begin{smat}
0&0&1&0\mst\\
0&1&-1&1\mst\\
1&0&0&0\mst\\
0&0&1&0\mst\\
\end{smat}
}="a6",
,(5,10.5)+(0,1.0)*++[F]{
\begin{smat}
0&0&1&0\mst\\
0&1&-1&1\mst\\
1&-1&1&0\mst\\
0&1&0&0\mst\\
\end{smat}
}="57",
,(3,12)+(0,1.0)*++[F]{
\begin{smat}
0&0&1&0\mst\\
0&1&-1&1\mst\\
0&0&1&0\mst\\
1&0&0&0\mst\\
\end{smat}
}="38",
,(7,12)+(0,1.0)*++[F]{
\begin{smat}
0&0&0&1\mst\\
0&1&0&0\mst\\
1&-1&1&0\mst\\
0&1&0&0\mst\\
\end{smat}
}="78",
,\ar@{-}"50";"31",\ar@{-}"50";"51",\ar@{-}"50";"71"
,\ar@{-}"31";"32",\ar@{-}"31";"52",\ar@{-}"51";"32"
,\ar@{-}"51";"72",\ar@{-}"24";"25",\ar@{-}"06";"37"
,\ar@{-}"71";"52",\ar@{-}"24";"45",\ar@{-}"06";"57"
,\ar@{-}"71";"72",\ar@{-}"44";"05",\ar@{-}"26";"17"
,\ar@{-}"32";"13",\ar@{-}"44";"65",\ar@{-}"26";"57"
,\ar@{-}"32";"33",\ar@{-}"44";"85",\ar@{-}"46";"17"
,\ar@{-}"32";"53",\ar@{-}"64";"25",\ar@{-}"46";"77"
,\ar@{-}"52";"53",\ar@{-}"64";"65",\ar@{-}"66";"37"
,\ar@{-}"72";"53",\ar@{-}"64";"a5",\ar@{-}"66";"97"
,\ar@{-}"72";"73",\ar@{-}"84";"45",\ar@{-}"86";"57"
,\ar@{-}"72";"93",\ar@{-}"84";"85",\ar@{-}"86";"77"
,\ar@{-}"13";"04",\ar@{-}"84";"a5",\ar@{-}"a6";"57"
,\ar@{-}"13";"24",\ar@{-}"a4";"a5",\ar@{-}"a6";"97"
,\ar@{-}"33";"04",\ar@{-}"05";"06",\ar@{-}"17";"38"
,\ar@{-}"33";"44",\ar@{-}"05";"26",\ar@{-}"37";"78"
,\ar@{-}"53";"24",\ar@{-}"25";"06",\ar@{-}"57";"38"
,\ar@{-}"53";"44",\ar@{-}"25";"86",\ar@{-}"57";"58"
,\ar@{-}"53";"64",\ar@{-}"45";"26",\ar@{-}"57";"78"
,\ar@{-}"53";"84",\ar@{-}"45";"46",\ar@{-}"77";"38"
,\ar@{-}"73";"64",\ar@{-}"45";"86",\ar@{-}"97";"78"
,\ar@{-}"73";"a4",\ar@{-}"65";"06",\ar@{-}"38";"39"
,\ar@{-}"93";"a4",\ar@{-}"65";"66",\ar@{-}"38";"59"
,\ar@{-}"93";"84",\ar@{-}"65";"a6",\ar@{-}"58";"39"
,\ar@{-}"04";"05",\ar@{-}"a5";"86",\ar@{-}"58";"79"
,\ar@{-}"24";"05",\ar@{-}"a5";"a6",\ar@{-}"78";"59"
,\ar@{-}"78";"79",\ar@{-}"39";"5a",\ar@{-}"59";"5a"
,\ar@{-}"79";"5a",\ar@{-}"85";"a6"
,\ar@{-}"85";"26"
\end{xy}
}
\end{center}
\end{figure}

\begin{center}
{\bf Acknowledgment.}\\
The author would like to thank the editor Nathan Reading and the anonymous referee for many helpful comments and suggestions to improve the manuscript.
\end{center}

\end{document}